\documentclass[12pt]{article}
\usepackage{amssymb,amsfonts,bezier,amstext,amsthm,amsmath}
\usepackage{color}
\usepackage{tikz}
\usetikzlibrary{fit,arrows,calc,positioning,decorations.markings}

\newcommand{\C}{\mathbb{C}}\newcommand{\R}{\mathbb{R}}\newcommand{\Z}
{\mathbb{Z}}\newcommand{\ov}{\overline}

\newcommand{\pa}{\partial}\newcommand{\Lb}{\Lambda}
\newcommand{\lb}{\lambda}\newcommand{\Om}{\Omega}

\newcommand{\pb}{\ov\partial}
\newcommand{\Pha}{\mathbb{P}^{0,1}}\newcommand{\ha}{{\cal H}^{0,1}(L^\ast)}
\newcommand{\Phh}{\mathbb{P}^{1,0}}\newcommand{\hh}{{\cal H}^{1,0}(L)}

\newcommand{\Div}{\text{div}}
\newcommand{\Deg}{\text{deg}}
\newcommand{\ord}{\text{ord}}
\newcommand{\DIV}{\text{Div}}
\newcommand{\cpl}{(\!(}\newcommand{\cpr}{)\!)}
\newcommand{\anl}{\langle\!\langle}\newcommand{\anr}{\rangle\!\rangle}
\newcommand{\vf}{\varphi}\newcommand{\vfs}{\varphi^\ast}
\newcommand{\om}{\omega}

\newcommand{\intl}{\int\limits}

\newtheorem{thm}{Theorem}[section]
\newtheorem{lem}[thm]{Lemma}
\newtheorem{cor}[thm]{Corollary}
\theoremstyle{remark}
\newtheorem{rem}[thm]{Remark}
\theoremstyle{definition}

\tikzset{->-/.style={decoration={
  markings,
  mark=at position .6 with {\arrow{>}}},postaction={decorate}}}
\tikzset{maina/.style={decoration={markings,mark=at position .6 with {\arrow{>}}},postaction={decorate},
circle,fill=white!50,draw,font=\sffamily\footnotesize}}
\tikzset{-<-/.style={decoration={
  markings,
  mark=at position .4 with {\arrow{<}}},postaction={decorate}}}

\title{Holomorphic Triples and the Prescribed Curvature
Problem on $S^2$}

\author{Alexandre C. Gon\c calves\footnote{Departamento de Computa\c c\~ao e Matem\'atica FFCLRP - USP, Av. Bandeirantes 
3900, 14040-901, Ribeir\~ao Preto, SP, Brasil. {\rm\bf acasa@ffclrp.usp.br}}}

\begin{document}

\maketitle



\begin{abstract}
We prove new results on existence of solutions for the prescribed gaussian
curvature problem on the euclidean sphere $S^2$. Those results are achieved by relating
this problem with the holomorphic triples theory on Riemann surfaces. We think this approach
might be applied to study some other semi-linear elliptic equations of $2^{\rm nd}$ order
on the sphere. 
\end{abstract}

\section{Introduction}

Let $M$ be a closed Riemann surface with metric $g_0$. By a pointwise conformal metric
we mean another metric $g$ given by dilation of $g_0$ by a positive smooth function.
Therefore, we can write $g=e^{2u}g_0$ for a a function $u\in C^{\infty}(M)$.
If $K_0$ and $K$ denote the gaussian curvatures of $g_0$
and $g$, respectively, it can be shown \cite{KW1}
\begin{equation}\label{curveq}
\Delta u+Ke^{2u}-K_0=0,
\end{equation}
where $\Delta$ denotes the Laplace-Beltrame operator on the metric $g_0$. Thus, 
finding a metric pointwise conformal to $g_0$ with curvature $K$ is equivalent
to finding classical solutions to the elliptic equation \eqref{curveq}.

This problem has been treated by several authors since the late 1960s 
\cite{Bg1,Bg2,KW1,KW2}. In \cite{KW1} Kazdan and Warner  obtained some general necessary and 
suficient conditions on the functions $K,K_0$ to assure existence of solutions to
\eqref{curveq}. They also found some non-existence conditions mainly in the
case of the euclidean sphere. 

On the other hand, it has long been known that equations like \eqref{curveq} are a particular 
case of the theory of holomorphic triples over K\"ahler manifolds \cite{Gp2,BrGp}. This theory grew out of 
the seminal work of Donaldson and Uhlenbeck-Yau about special metrics on stable vector bundles, which developed into 
an active area of work since the 1980s \cite{Do1,Do2,UY,Si,Br1,Br2,Gp2}.
The Vortex equation was introduced in \cite{Br1} and evolved into the holomorphic triples theory \cite{Gp2,BrGp},
where not only holomorphic vector bundles, but also prescribed cohomology classes on the bundles are considered.

In \cite{Go2} the study of equation \eqref{curveq} is presented in connection with the vortex and holomorphic triples 
theory, by means of two distinct though related problems:
\begin{align}
\Delta u + |[\phi]|_u^2-\lambda =0 \label{phieq},\\
\Delta u + |[\eta]|_u^2-\lambda =0 \label{etaeq}.
\end{align}

Equations \eqref{phieq} and \eqref{etaeq} are defined on a closed Riemann surface $M$ for a real parameter $\lb>0$ 
and for \emph{cohomology classes} $[\phi]$ and $[\eta]$ living, respectively, in the cohomology complex of
holomorphic line bundles $L$ and $L^\ast$ over $M$. 
The terms $|[\phi]|_u^2$ and $|[\eta]|_u^2$ refer to the pointwise squared norm of representatives of these
classes, in a hermitian metric given by dilation of the original metric by a factor $e^{2u}$. The function $u$ is 
a real smooth function on $M$ and is meant to be the unknown in the equations. 

In the prescribed curvature problem presented by equation \eqref{curveq} one is often interested in the 
case $K_0\equiv$constant. Since the work of Kazdan and Warner this is already well known for all surfaces 
with non-positive Euler characteristic, as well as for the projective plane $\mathbb{PR}^2$. Despite some non-trivial 
non-existence conditions have been found, the case of $M=S^2$ is where most open questions remain. 
It amounts to say that up to our knowledge, all results on existence for \eqref{curveq} after \cite{KW2} play
on several suficient conditions for the function $K$, one of them being $K>0$ \cite{CY,JunW}. Existence for 
\eqref{curveq} is also known when $K$ is symmetric about the origing (considering the cannonical 
inclusion $S^2\hookrightarrow\R^3$), after the work \cite{Ms}.

Our results apply for functions $K$ which are the squared modulus of holomorphic
sections, typically having some zeros, and not necessarily symmetric about the origin. 
Most importantly, those results can only be established after we
explicitly connect equations \eqref{phieq} and \eqref{etaeq}, and strongly rely on algebraic-geometric 
elements of the involved bundles, like their Chern classes. We conjecture that this algebraic fact we use for 
studying equation \eqref{curveq} might be applied even for more general functions $K$, and has not been 
pursued by other authors so far.

A brief description of this work: in section 2 we collect some well-known facts on the theory of line
bundles over riemann surfaces, as well as results on metric equations like the vortex equation; in section
3 we prove the main results necessary to understand the cohomology classes of the dual bundle $L^\ast$ from
the analytical viewpoint, contained in Lemmas \ref{meromorphic-on-E}, \ref{sol_delbar-eq} and \ref{calc-deg};
and in section 4 we apply those results to show existence or non-existence of solutions for \eqref{curveq},
for some conformal curvatures $K$, which are summarized by Lemma \ref{non-exist-lem} and Theorems 
\ref{exist-top-arg} through \ref{existence_lem-ger}.

I would like to acknowledge the encouragement of my department colleagues while this manuscript was written,
and also the comments of the referees that improved the presentation of this work.

\section{Basics on the Geometry of Holomorphic Bundles}

This section is only meant to set up notation. For a deep study throughout these matters we recommend \cite{GH,KN}.

\subsection{Hermitian Bundles and Cohomology}

Let $E$ be a smooth complex vector bundle over a complex manifold $X$. Associated to $E$ we have the
dual bundle $E^\ast$, conjugate bundle $\ov E$ and endomorphism bundle $End(E)$. A (hermitian) metric is 
then a smooth isomorphism $H:E\to\ov E^\ast$ which is positive definite in each fiber. The bundle $E$ together 
with the structure given by $H$ is a \emph{hermitian bundle}.

Denote by $(T^\ast X)_\C$ the complexified cotangent bundle of $X$, which splits as $(T^\ast X)_\C=T^{1,0}X\oplus T^{0,1}X$. 
The bundle $T^{1,0}X$ is the holomorphic cotangent bundle of $X$ (home of the famous ``holomorphic differentials'').
Let $\Lb^{p,q}T^\ast X=\Lb^pT^{1,0}X\otimes\Lb^qT^{0,1}X$ for non-negative integers $p,q$, and let $\Gamma(\cdot)$ be 
the functor that takes a bundle to its space of smooth sections. We set $\Om^{p,q}(E)=\Gamma(\Lb^{p,q}T^\ast X\otimes E)$.
Any $\phi\in\Om^{p,q}(E)$ is a smooth section of holomorphic type $(p,q)$ and values in $E$. 

A \emph{holomorphic structure} on $E$ is an operator $D'':\Om^{p,q}(E)\longrightarrow\Om^{p,q+1}(E)$
that satisfies $(D'')^2=0$ and enjoys some typical properties of a covariant derivative (see \cite{GH,KN}).
Indeed, a \emph{connection} is a covariant derivative $D:\Om^m(E)\longrightarrow\Om^{m+1}(E)$, where 
$\Om^m(E)=\oplus_{p+q=m}\Om^{p,q}(E)$. Any connection decomposes after the splitting of the cotangent
bundle $D=D^{1,0}+D^{0,1}$. It is well known that for a given hermitian metric $H$ and holomorphic structure $D''$ 
there is only one connection $D=D_{H,D''}$ compatible with both, which means, $D^{0,1}=D''$ and $D(H)=0$.
This connection is called the \emph{Chern connection}.

The \emph{curvature} of a connection $D$ is the compound $F_D=D^2:\Om^m(E)\to\Om^{m+2}(E)$, which is a 2-form 
section of the bundle $End(E)$. An important topological invariant associated to the bundle $E$ is its 
\emph{first Chern class} $\frac i{2\pi}[tr(F_D)]$ which is a cohomology class on the base manifold 
(here $tr(\cdot)$ is the trace of the endomorphism coefficient of $F_D$ and $i=\sqrt{-1}$).
The curvature $F_D\in\Om^2(End(E))$ of any Chern connection has only the $(1,1)$ component, so that 
$F_D=D''\circ D'+D'\circ D''$ (we denote $D'=D^{1,0}$ from now on).

Since $(D'')^2=0$ we get a cochain complex $(\Om^{p,q}(E),D'')$ whose cohomology we denote
\begin{equation*}
  {\cal H}^{p,q}(E)=\frac{ker\,D'':\Om^{p,q}(E)\to\Om^{p,q+1}(E)}{im\,D'':\Om^{p,q-1}(E)\to\Om^{p,q}(E)}\ .
\end{equation*}

A \emph{holomorphic section} is any $\phi\in\Om^{p,q}(E)$ such that $D''\phi=0$. Similarly, an \emph{anti-holomorphic
section} is any section $\eta$ solving $D'\eta=0$, for a given $D'$ operator.

Let $\phi\in\Om^{p,q}(E)$, we set the $H$-dual of $\phi$ as $\phi^{\ast H}=\ov{H(\phi)}\in\Om^{q,p}(E^\ast)$. 
The $H$-dual of a section valued form is obtained by conjugating the form part and dualizing, in the usual way,
the bundle coefficient. If the connection $D$ is hermitian then a section $\phi$ is holomorphic if and only if
$\phi^{\ast H}$ is anti-holomorphic. 

\subsection{Line Bundles, Degrees and Divisors over Surfaces}

We now turn our attention to the case of a closed oriented Riemann surface $X=M$. 

Recall that a \emph{meromorphic section} on the holomorphic bundle $E$ over $M$ is a holomorphic
section $\phi$ on $M-\{x_1,\dots,x_t\}$ and such that in a neighborhood of each $x_j$, $\phi=z_j^{m_j}\zeta_j$,
where $z_j$ is a holomorphic local coordinate on $M$ with $z_j(x_j)=0$ and $\zeta_j$ is a regular 
holomorphic local section. The \emph{divisor} of $\phi$ is the formal linear combination $\DIV(\phi)=\sum_{j=1}^tm_j.x_j$,
and the degree of $\phi$ is $\Deg(\phi)=\sum_{j=1}^tm_j$. The integer $m_j$ is the \emph{order} of $\phi$ at $x_j$, 
$m_j=\ord_{x_j}(\phi)$.

A \emph{line bundle} is a holomorphic bundle $L$ of rank 1 over $M$. It can be shown that any line bundle
has a non-vanishing meromorphic section $\phi$ (\cite{DV}), and we set $\Deg(L)=\Deg(\phi)$. Since the endomorphism bundle of 
$L$ is just the trivial bundle $M\times\C$, and the curvature reduces to a closed 2-form on $M$, we get an
analitycal way of computing its degree,
\begin{equation}
  \deg(L)=\int_M\,\frac i{2\pi}F_D.
\end{equation}
Observe that $\hh$ is the space of holomorphic sections of the bundle $T^{1,0}(M)\otimes L$, hence to avoid it 
to be trivial we always assume 
\begin{equation}\label{degL>eulerM}
\Deg(L)\geq-\Deg(T^{1,0}(M)). 
\end{equation}

Of great interest to us are the cohomologies ${\cal H}^{1,0}(L)$ and ${\cal H}^{0,1}(L^\ast)$. Clearly ${\cal H}^{1,0}(L)$
is identified with the set of holomorphic sections. On the other hand any section on $L^\ast$ of holomorphic type $(0,1)$
is $D''$-closed, and so represents a cohomology class in ${\cal H}^{0,1}(L^\ast)$. By standard Hodge Theory \cite{GH}
any class on ${\cal H}^{0,1}(L^\ast)$ has exactly one harmonic representative, which must be $H$-antiholomorphic, hence the map 
\begin{equation}\label{ha-map-hh}
  \ast H:{\cal H}^{1,0}(L)\longrightarrow{\cal H}^{0,1}(L^\ast)
\end{equation}
is an anti-isomorphism between these two vector spaces.

By wedging the 1-forms we define a bilinear operator $\Om^{1,0}(L)\times\Om^{0,1}(L^\ast)\to\Om^2(\C)$ taking 
sections $\phi$ and $\eta$ to $(\phi\wedge\eta)$, and a coupling 
\begin{equation}\label{coupling}
\cpl\phi,\eta\cpr=\int_M i(\phi\wedge\eta).
\end{equation}
Because of Stokes' Theorem and integration by parts this coupling descends to cohomology classes, so that
$\cpl[\phi],[\eta]\cpr=\cpl\phi,\eta\cpr$, as long as $\phi$ and $\eta$ represent classes $[\phi]\in\hh$ and 
$[\eta]\in\ha$, respectively. Similarly, we have a coupling given by the metric $H$ by setting
$\anl\phi,\psi\anr_H=\cpl\phi,\psi^{\ast H}\cpr$, for any $\phi,\psi\in\Om^{1,0}(L)$.

Two metrics $H$ and $H_0$ on $L$ are related by a  positive dilation in each fiber, so that $H=H_u=H_0e^{2u}$
for a smooth function $u$ on $M$. For a section $\phi$ on $L$ we have 
$|\phi(x)|_{H_u}^2=|\phi(x)|_u^2=|\phi(x)|_0^2e^{2u(x)}$ for any $x\in M$, and for a section $\eta$ on $L^\ast$ 
it holds $|\eta(x)|_u^2=|\eta(x)|_0^2e^{-2u(x)}$, for the metric on $L^\ast$ is set by duality. 
The curvatures $F_{H_u}$ and $F_{H_0}$ associated to the correspondent Chern connections are related by 
$i\Lb F_{H_u}=i\Lb F_{H_0}-\Delta u$, where $\Lb$ is the contraction with the volumn element $\nu$ on $M$
and $\Delta$ is the Laplace-Beltrame operator on functions. Assuming that $|M|=1$ and $H_0$ is a metric yielding 
constant curvature $i\Lb F_{H_0}=2\pi\deg(L)$ we obtain
\begin{equation}\label{Fu-F0}
  i\Lb F_{H_u}=2\pi\Deg(L)-\Delta u.
\end{equation}

We restrict to the case of the euclidean sphere $M=S^2$, but we dilate the standard metric by a constant factor
so that its gaussian curvature is $4\pi$ and $|S^2|=1$. If $x\in S^2$ is any point we set $z=z_x:S^2-\{x\}\to\C$ 
as a stereographic projection with north pole at $x$. 

Two important facts about bundles over the sphere are stated below. The first one comes from a simple computation
using a stereographic coordinate, while the proof for the second can be found in \cite{Gu}. 
\begin{lem}\label{facts_on-sphere}
Let $M=S^2$ be the base manifold. Then\\
(a) Line bundles are classified by their degrees.\\
(b) Any rank 2 holomorphic bundle $E$ splits holomorphically as the sum of line bundles, $E=L_1\oplus L_2$. 
\end{lem}

Let $N=(0,0,1)$ and $z=z_N$. Any line bundle $L$ has a trivialization over $S^2-\{N\}$ given by a
``cannonical'' meromorphic section $\zeta_L$ whose singular set is $\{N\}$. By using the coordinate 
$w=1/z$ we get a section $\zeta_{L,S}$ holomorphic and regular over $S^2-\{S\}$, where $S=-N$. The
gauge transformation between them is
\begin{equation}
  \zeta_L(x)=w(x)^{\Deg(L)}\,\zeta_{L,S}(x)\qquad \text{for all}\ x\in S^2-\{S,N\}.
\end{equation}

If $\deg(L)\geq0$ an arbitrary holomorphic section of $L$ is given by $h\zeta_L$ for some polynomial $h=h(z)$ of degree 
bounded by $\deg(L)$. The bundle $T^{1,0}S^2$ is spanned by the holomorphic differential $\zeta_{T^{1,0}S^2}=dz$, hence
the set $\hh$ of holomorphic sections of $\Om^{1,0}(L)=\Gamma(L\otimes T^{1,0}S^2)$ consists of sections $\phi$ 
of the form $\phi=g\,\zeta_L\,dz$, where $g(z)$ is a polynomial with degree less than or equal to $\deg(L)-2$.

In coordinates, the $H$-dual of $\phi$ is the anti-holomophic section $\eta=\ov g\,\zeta_L^{\ast H}d\ov z$, and
the section $\zeta_L^{\ast H}$ can be written 
\begin{equation}
\zeta_L^{\ast H}=\frac{\zeta_{L^\ast}}{|\zeta_{L^\ast}|_H^2}=|\zeta_L|_H^2\,\zeta_{L^\ast}\ .
\end{equation}
For later use we express the $H$-norm of $\eta$ in the $z$ and $w$ coordinates:
\begin{equation}\label{eta-Hnorm}
  |\eta|_H^2=|g|^2|\zeta_L|_H^2|dz|^2=|g|^2|w|^{2(\Deg(L)-2)}|\zeta_{L,S}|_H^2|dw|^2.
\end{equation}

The next Lemma helps us to find an explicit expression for $|\zeta_L|_{H_0}^2$. 
\begin{lem}\label{curvature-merom}
Let $\zeta$ be a meromorphic section on $L$ and $H$ be a metric. Then in any open region where $\zeta$ is regular
the $H$-curvature of $L$ is given by $i\Lb F_H=-\Delta\ln|\zeta|_H$.
\end{lem}

Taking the cannonical section $\zeta_L$ we get $\Delta\ln|\zeta_L|_{H_0}(x)=-2\pi\Deg(L)$ for all
$x\in S^2-\{N\}$. An inspection shows that $\Delta[\frac{\Deg(L)}2\ln(1+|z|^2)]=2\pi\Deg(L)$, thus
if we set 
\begin{equation}\label{H0.def}
  |\zeta_L|_{H_0}^2(x)=(1+|z(x)|^2)^{-\Deg(L)}\qquad\text{for}\ x\neq N
\end{equation}
we get a prospective function describing the metric $H_0$. To make sure it works fine we notice
that in the other trivialization
\begin{equation}\label{H0.def-alt}
|\zeta_{L,S}|_{H_0}^2(x)=|w(x)^{-\Deg(L)}\zeta_L(x)|_{H_0}^2=\frac1{(|w(x)|^2+1)^{\Deg(L)}}\qquad\text{for}\ x\neq S,
\end{equation}
so $H_0$ is smooth at each fiber of $L$. If any other metric $H_u$ yields constant curvature to
$L$ then by equation \eqref{Fu-F0} it holds $\Delta u=0$, so $u$ is a constant and $H_u$ is just a uniform
dilation of the given $H_0$. We set equation \eqref{H0.def} (or \eqref{H0.def-alt}) as the definition for
$H_0$.

\subsection{Holomorphic Extensions and Stability}

Recall that a holomorphic extension of $E_2$ by $E_1$ is a short exact sequence of holomorphic bundles 
and morphisms $e:\ 0\to E_1\to E\to E_2\to0$ over the same base manifold. There is a natural concept of
\emph{isomorphism of extensions}, and we define $Ext(E_2,E_1)$ as the set of classes of isomorphic extensions.
Also let $Hom(E_2,E_1)=E_1\otimes E_2^\ast$ be the bundle of homomorphisms $E_2\to E_1$. The proof of the next Lemma can 
be found in \cite{Ko,GH}. 
\begin{lem}\label{cohom-ext-corresp}
There is a natural one-to-one correspondence between $Ext(E_2,E_1)$ and ${\cal H}^{0,1}(Hom(E_2,E_1))$.
\end{lem}  

For a holomorphic bundle $E$ we define
\begin{equation}
\Div(E)=\sup\,\{\Deg(J)\,|\,J\subset E\ \text{is a holomorphic line subbundle}\}.
\end{equation}
It is well known that $\Div(E)$ is finite on Riemann Surfaces \cite{Gu}. Taking $[\eta]\in{\cal H}^{0,1}(Hom(E_2,E_1))$ we 
can define $\Div[\eta]=\Div(E)$ where $E$ is the middle term of the extension associated to $[\eta]$.

Now fix $L_1,L_2$ line bundles over $S^2$ and consider holomorphic extensions $0\to L_1\to E\to L_2\to 0$. 
Therefore $E$ is a rank 2 vector bundle which is \emph{topologically}, but not \emph{holomorphically} in general, 
the direct sum of $L_1$ and $L_2$. We set from now on the bundle $L=L_2\otimes L_1^\ast$. 
Hence the set of extensions of $L_2$ by $L_1$ is just ${\cal H}^{0,1}(L^\ast)$. 

\begin{lem}\label{div-eta<=degL2}
Let $[\eta]\in\ha$. Then ${\rm div}[\eta]\leq\max\,\{{\rm deg}(L_1),{\rm deg}(L_2)\}$. In case 
${\rm deg}(L_2)>{\rm deg}(L_1)$ equality holds if and only if $[\eta]=0$.
\end{lem}
\begin{proof}
Let 
\begin{equation}\label{extension_eta}
\qquad 0\to L_1\to E\overset\pi\to L_2\to 0
\end{equation}
be the extension associated to $[\eta]$. Let $J\subset E$ be a holomorphic line subbundle and
$\phi$ be a (non-trivial) meromorphic section of $J$. Consider first that $\pi(\phi)\equiv 0$. 
Then the range of $\phi$ lies within $L_1$, hence $J=L_1$ and $\Deg(J)=\Deg(L_1)$.

Now suppose $\pi(\phi)\not\equiv0$. Then $\pi(\phi)$ is a meromorphic section on $L_2$. If $x\in M$
let $z$ be a holomorphic coordinate with $z(x)=0$. Let $\zeta$ be a regular holomorphic section on J in
a neighborhood of $x$. Then $\phi(z)=h(z)\zeta(z)$ for $z$ close to $0$, where $h(z)$ is some local meromorphic
function. Clearly $\ord_x(\phi)=\ord_x(h)\leq\ord_x(h)+\ord_x(\pi(\zeta))=\ord_x(\pi(\phi))$, since 
$\ord_x(\pi(\zeta))\geq0$. We conclude that 
\begin{equation}
\Deg(\pi(\phi))=\sum_{x\in M}\ord_x(\pi(\phi))\geq\sum_{x\in M}\ord_x(\phi)=\Deg(\phi).
\end{equation}
Therefore $\Deg(L_2)\geq\Deg(J)$. Considering both cases we arrive at $\Deg(J)\leq\max\,\{\Deg(L_1),\Deg(L_2)\}$.
From this and the arbitrarity of $J$ the first assertion of the Lemma follows.

Now if $[\eta]=0$ the extension \eqref{extension_eta} is trivial, hence $L_2\hookrightarrow E$ holomorphically.
Thus $\Div(E)\geq\max\{\Deg(L_1),\Deg(L_2)\}$. The first part of the proof already gave us the reversal inequality,
and we obtain $\Div[\eta]=\max\{\Deg(L_1),\Deg(L_2)\}$.

Finally assume $\Deg(L_2)>\Deg(L_1)$ and $\Div[\eta]=\max\{\Deg(L_1),\Deg(L_2)\}=\Deg(L_2)$.
Let $J\subset E$ be a holomorphic line subbundle such that $\Deg(J)=\Deg(L_2)$. Since $J\neq L_1$ 
the restriction map $\pi|_J:J\to L_2$ is a non-trivial holomorphic morphism. Further $\pi|_J$ is a
section of $L_2\otimes J^\ast$, and $\Deg(L_2\otimes J^\ast)=\Deg(L_2)-\Deg(J)=0$, hence $\pi|_J$ 
has no zeros. This is equivalent to saying that $E=L_1\oplus J$ holomorphically. It is straightforward
to check that the trivial extension $0\to L_1\to E\to J\to 0$ is isomorphic to extension \eqref{extension_eta}.
The consclusion is that $[\eta]=0$, and the second assertion of the Lemma is proven.
\end{proof}

\subsection{The metric equations}

Let $0\to E_1\to E\to E_2\to 0$ be an extension, thus $E_1\subset E$ is subholomorphic and $E_2=E/E_1$.
A metric $H=H_E$ induces metrics $H_j$ on $E_j$ and an identification $E_2\sim E_1^{\perp H}$.
Respect to the orthogonal decomposition $E=E_1\oplus E_2$ we can write the equation
\begin{equation}\label{hol-ext}
    i\Lb F_H = \left(\begin{array}{cc} \ \tau_1\  & 0 \\ 0 & \ \tau_2\  
  \end{array}\right)\ ,
\end{equation}
where the right-hand-side is a section of $End(E)$, assembled as a weighted combination of the orthogonal projections 
$E\to E_j$, $j=1,2$, for some real constants $\tau_1,\tau_2$, and $\Lambda$ is the contraction with the K\"ahler form
on the surface.

The problem stated by equation \eqref{hol-ext} is a particular case of the holomorphic and cohomology triples problems,
which have been introduced in \cite{Gp2} and \cite{BrGp}. These problems constitute a generalization of the
Hermite-Einstein equation over K\"ahler manifolds \cite{UY}. The typical theorem in those theories, known as
\emph{the Hitchin-Kobayashi correspondence}, states that a solution exists for the metric equations as long as
an algebraic condition called \emph{stability} (or \emph{polystability} in a more general case) is satisfied for the 
involved bundles and perhaps some other structures, like prescribed sections or cohomology classes.

To properly express this theory, whose details can be found in \cite{BrGp,Gp2}, we would need to elaborate 
on definitions and notation that go far beyond the line of our article. Instead, we'd rather state in a 
summary what concern to us. 
Since we have line bundles $L_1,L_2$ over the riemann surface $M$ an extension of $L_2$ by $L_1$ is some 
$[\eta]\in\ha$. If $\alpha\in\R$ we define the $\alpha$-slope of $[\eta]$,
\begin{equation}
\mu_\alpha([\eta])=\frac{\Deg(L_1)+\Deg(L_2)+\alpha}2.
\end{equation} 
Then we say that $[\eta]$ is $\alpha$-stable if 
\begin{equation}\label{alpha-stable}
\max\{\Deg(L_1),\Div[\eta]+\alpha\}<\mu_\alpha([\eta]).
\end{equation}
From inequality \eqref{alpha-stable} and the definition of $\mu_\alpha$ we conclude that $[\eta]$ is
$\alpha$-stable if and only if $\Deg(L_1)-\Deg(L_2)<\alpha<\Deg(L_1)+\Deg(L_2)-2\Div[\eta]$. A necessary
condition for $\alpha$-stability is then that a strict inequality happen between the first and the third members
of the latter. The next theorem replicates the results from Proposition 3.8 and Theorem 3.9 of \cite{BrGp}.
\begin{thm}\label{estavel} 
Let $\tau_1$ and $\tau_2$ be real numbers such that $\tau_1+\tau_2=2\pi({\rm deg}(L_1)+{\rm deg}(L_2))$. 
Let $\alpha=\frac1{2\pi}(\tau_1-\tau_2)<0$ and assume $[\eta]\neq0$. Then there is a metric $H_E$ 
satisfying \eqref{hol-ext} if and only if $[\eta]$ is $\alpha$-stable.
\end{thm}
\begin{rem}
Proposition 3.8 of \cite{BrGp} skips the condition $\alpha<0$. That is actually necessary to derive 
the $\alpha$-stability in case the metric solution $H_E$ exists. By the way, there is a straightforward example  
of a solution $H_E$ for \eqref{hol-ext} in an extension over $S^2$, where any $\alpha\geq0$ is allowed, hence
outside the admissible range of $\alpha$-stability as defined by \eqref{alpha-stable}.
\end{rem}

Let $H=H_E$ be a metric satisfying \eqref{hol-ext} for an extension $[\eta]$ of line bundles. 
The metric connection on $E$ is then 
\begin{equation}
     D_E = \left(\begin{array}{cc} D_1 & A \\ -A^{\ast H} & D_2 \end{array}\right)\ ,
\end{equation}
for $D_j$ being the Chern connections on $L_j$ and $A$ being the second fundamental form of the
inclusion $L_1^{\perp H}\hookrightarrow E$. Computing $F_H=D_E^2$ and substituting into equation 
\eqref{hol-ext} we find the system
\begin{equation}\label{eqmot}
\left\{\begin{aligned} 
     i\Lb F_1 - i\Lb A\wedge A^{\ast H} & = \tau_1 \\
     i\Lb F_2 - i\Lb A^{\ast H}\wedge A & = \tau_2 \\
        D(A) & = 0\ .
\end{aligned}\right.\end{equation}
The form $A$ has holomorphic type $(0,1)$ since $L_1^{\perp H}\hookrightarrow E$ is antiholomorphic.
Indeed, $A$ is $D''$-closed, and its cohomology class is the one given in the beginning,
$[A]=[\eta]$. The third equation in \eqref{eqmot} implies that $A$ is antiholomorphic, and we can assume 
from now on that $A=\eta+D''\xi$ is the $H$-antiholomorphic representative of the class $[\eta]$.

Making $\lb=2\pi\Deg(L)+\tau_1-\tau_2$ and following a computation similar to \cite{Go2} (equations (3.8)-(3.10))
we obtain from \eqref{eqmot}
\begin{equation}\label{system-eta}
\left\{\begin{aligned} &\Delta u+2|\eta+D''\xi|_0^2e^{-2u}-\lb=0\\ &D'_u(\eta+D''\xi)=0\ ,\end{aligned}\right.
\end{equation}
where $u$ is the function associated to the pointwise metric change in $L,L^\ast$.

Therefore, a solution $H_E$ for \eqref{hol-ext} gives us a smooth function $u$ on $S^2$ and a section $\xi$ of $L^\ast$ 
that solve \eqref{system-eta}. Reciprocally, given a pair $(u,\xi)\in C^\infty(S^2)\times\Om^0(L^\ast)$ that 
solves \eqref{system-eta} it is straightforward to obtain the correspondent solution $H_E$ for \eqref{hol-ext} 
(see \cite{Go2}).

The dual problem of \eqref{system-eta} is stated as follows: using the $H$-identification given by \eqref{ha-map-hh}
we set $\phi=(\eta+D''\xi)^{\ast H_u}\in\hh$. Clearly $|\phi|_{H_u}^2=|\eta+D''\xi|_{H_u}^2$, hence
\eqref{system-eta} becomes equivalent - for the particular solution $u$ - to system
\begin{equation}\label{system-phi}
\left\{\begin{aligned} &\Delta u+2|\phi|_0^2e^{2u}-\lb=0\\ &D''(\phi)=0.\end{aligned}\right.
\end{equation}
\begin{rem}
The first of equations \eqref{system-eta} and \eqref{system-phi} carry a factor 2 which had been absorbed 
in equations \eqref{phieq} and \eqref{etaeq} (see also equations 3.6 and 3.8 in \cite{Go2}). 
\end{rem}

We take a minute to compare problems \eqref{system-eta} and \eqref{system-phi}. They are not quite
the same because varying the real parameter $\lb$ and keeping $[\eta]$ fixed will vary the metric,
and so will change the correspondent $[\phi]=[\eta]^{\ast H_u}$. It might sound that \eqref{system-phi}
is more likely to the taste of the analyst, because a fixed $[\phi]$ has only one representative regardless
of the metric and one has a shape for the term $|\phi|_0^2$. On the other hand the same class $[\eta]$
has different representatives for different metrics, making the sight of the term $|\eta+D''\xi|_0^2$ a bit
obscure.

Nevertheless, equations \eqref{system-phi} lose an important characteristic that is enjoyed by \eqref{system-eta}: its
linearization is not sign definite. This is roughly accounted for the difference in sign of the exponents in
$e^{2u}$ and $e^{-2u}$ of either one. Further, system \eqref{system-eta} has the results on extensions holding
in the range $0<\lb<2\pi\Deg(L)$. The translation of results from \eqref{system-eta} to \eqref{system-phi}
for some cases of $[\phi]$ is one of the main targets of this work.

Observe that we can refer to a solution $(u,\xi)$ for \eqref{system-eta} simply by $u$, since there is only one 
section $\xi=\xi(u)$ satisfying the second of equations \eqref{system-eta}. 
\begin{lem}{\rm(\cite{Go2} Theorem 4.7 and Corollary 4.9)}\label{exist-sol-eta}
Let $0<\lb_0<2\pi\,{\rm deg}(L)$. Assume there is a solution $u=u_0$ for \eqref{system-eta} in the parameters
$[\eta]=[\eta_0]$ and $\lb=\lb_0$. Then this solution is unique. There is a neighborhood
$U\times(\lb_0-\varepsilon,\lb_0+\varepsilon)\subset\ha\times\R$ of $([\eta_0],\lb_0)$ and a smooth
map $u:U\times(\lb_0-\varepsilon,\lb_0+\varepsilon)\to C^\infty(S^2)$ taking parameters to solutions
of \eqref{system-eta}.
\end{lem}

\begin{lem}\label{existence-terms-diveta}
Let $[\eta]\in\ha-\{0\}$. Then $0<4\pi({\rm deg}(L_2)-{\rm div}[\eta])\leq2\pi\,{\rm deg}(L)$. System 
\eqref{system-eta} has a solution for any $\lb\in(0,4\pi({\rm deg}(L_2)-{\rm div}[\eta]))$ and no solution
for $\lb\in[4\pi\,({\rm deg}(L_2)-{\rm div}[\eta]),2\pi\,{\rm deg}(L))$.
\end{lem}
\begin{proof}
Inequality $0<4\pi(\Deg(L_2)-\Div[\eta])$ holds because of $[\eta]\neq0$ and Lemma \ref{div-eta<=degL2}. 
For the other inequality we pick an extension $0\to L_1\to E\to L_2\to0$ representing $[\eta]$. From lemma
\ref{facts_on-sphere} part (b) we have $E=\tilde{L_1}\oplus\tilde{L_2}$ holomorphically, for some
line bundles $\tilde{L_1},\tilde{L_2}$ with $\Deg(\tilde{L_2})\geq\Deg(\tilde{L_1})$. The general theory of
Chern classes gives us $\Deg(E)=\Deg(L_1)+\Deg(L_2)=\Deg(\tilde{L_1})+\Deg(\tilde{L_2})$. Clearly
$\Div[\eta]=\Deg(\tilde{L_2})\geq1/2(\Deg(L_1)+\Deg(L_2))$, hence $4\pi(\Deg(L_2)-\Div[\eta])\leq2\pi\Deg(L)$.

The second assertion of the Lemma comes from the $\alpha$-stability condition restated for the parameter 
$\lb$ combined with Theorem \ref{estavel} and the subsequent discussion.  
\end{proof}

\section{The Space of Extensions over $S^2$.}

After the results in the previous section it becomes relevant to understand the space $\ha$,
which is non-trivial because of \eqref{degL>eulerM}. 

Let $[\eta]\in\ha$ and $H$ be a metric. Let $\eta$ be the $H$-antiholomorphic representative for $[\eta]$.
We aim to compute $\Div[\eta]$. For that sake consider an extension $0\to L_1\to E\to L_2\to 0$ for
$[\eta]$. After Lemma \ref{cohom-ext-corresp} we can take $E=L_1\oplus_{top}L_2$ with holomorphic operator 
$D''_\eta$, where 
\begin{equation}
D_\eta''=\left(\begin{array}{cc} D_1'' & \eta \\ 0 & D_2''\end{array}\right).
\end{equation}
We denote the cannonical meromorphic section of $L_j$ by $\zeta_j$,
$j=1,2$. The investigation of the meromorphic sections of $E$ starts with the
\begin{lem}\label{meromorphic-on-E}
There is a smooth $f:S^2\to\C$ such that $\psi=(f\zeta_1,\zeta_2)$ is a meromorphic section of $E$.
If $J\subset E$ is any line bundle not equal to $L_1$ then there is a meromorphic function $h$ on
$S^2$ such that $\tilde\psi=((f+h)\zeta_1,\zeta_2)$ is meromorphic and spans $J$.
\end{lem}

In the sequel we will write $k=\Deg(L)$. The proof of Lemma \ref{meromorphic-on-E} will come straight after the next result. 
\begin{lem}\label{sol_delbar-eq}
Write $\eta=\ov g\,\zeta_L^{\ast H}d\ov z$. Then there is a single $f\in C^\infty(S^2)$ with values in $\C$, 
and such that 
\begin{equation}\label{delzbar-f=}
\pb_z f=-\ov g\,|\zeta_L|_H^2
\end{equation}
and $f(N)=0$, $N$ the north pole. This function can be written as $f={\cal O}-p_f$ in a neighborhood
of $N$, with $p_f$ a polynomial in $w=\frac 1z$ of degree smaller than $k$ and ${\cal O}$ is a local
smooth function such that $|{\cal O}(w)|\leq C|w|^k$, for some constant $C>0$. 
Still, $p_f(w)=\sum_{j=1}^{k-1}b_jw^j$ and
\begin{equation}\label{bj-coeff}
  b_j=\int_{S^2}\frac{|\eta|_H^2}g\,w^{-j+1}\,\nu(w), \qquad 1\leq j\leq k-1.
\end{equation}
\end{lem}
\begin{proof}
Let $h\in C^\infty(S^2)$. Define a function $f:S^2-\{N\}\to\C$ by
\begin{equation}\label{f-int-sphere}
  f(z)=\int_{S^2}\frac{h(z')}{z-z'}\,\nu(z').
\end{equation}
We claim that $f$ is actually defined in the whole of $S^2$, is $C^\infty$  and it holds $\pb_z f(z)=h(z)/(1+|z|^2)^2$
away from the north pole.
Indeed, we can rewrite the integral in \eqref{f-int-sphere} as an integral in the plane
\begin{equation}\label{f-int-plane}
f(z)=\frac1{2\pi i}\int_{\R^2}\frac{h(z')}{(1+|z'|^2)^2(z'-z)}\,dz'\wedge d\ov z'.
\end{equation}
Recall that by the $\ov\partial$-Poincar\'e Lemma \cite{GH} the integral on the right-hand side of \eqref{f-int-plane}
is a function in the parameter $z$ whose $\pb_z$-derivative equals $h(z)/(1+|z|^2)^2$. In spite of the integral
in \cite{GH} be performed in a bounded region of the plane the argument for the derivative requires a local
computation and still holds in our case. Finally, changing in \eqref{f-int-sphere} the coordinate $z'$ by
$w'=1/z'$ we get
\begin{equation}
  f(w)=\int_{S^2}\frac{h(w')}{w'-w}ww'\,\nu(w'),
\end{equation}
from what we can see $f$ is well defined and smooth in $N$ $(w=0)$, as well as $f(N)=0$. If $\tilde f$ is
any function smooth on $S^2$ and such that $\pb_z\tilde f=\pb_z f$ we have $\tilde f-f$ holomorphic in
$S^2-\{N\}$, and hence constant because $\tilde f-f$ is bounded. We conclude there is exactly one $f$ 
satisfying $\pb_z f(z)=h(z)/(1+|z|^2)^2$.

Clearly the first part of the Lemma follows if we take $h(z)=-\ov g\,|\zeta_L|_H^2(1+|z|^2)^2$, observing
that this choice makes $h$ smooth: $h(z)=-g(z)^{-1}|\eta(z)|_H^2=-\ov{g(z)}|\zeta_{L\otimes T^{1,0}S^2}(z)|_H^2$.
The function $f$ given by \eqref{f-int-sphere} becomes
\begin{equation}\label{f-exp-eta}
  f(w)=\int_{S^2}\frac{|\eta|_H^2}g\,\frac{ww'}{w-w'}\,\nu(w').
\end{equation}

Using identity \eqref{eta-Hnorm} we see that the integrands on \eqref{bj-coeff} can be written, 
in each trivialization $\zeta_L$ or $\zeta_{L,S}$, as
\begin{equation*}
\frac{|\eta|_H^2}g\,w^{-j+1}=\ov g|\zeta_L|_H^2|dz|^2w^{-j+1}=(\ov g\,\ov w^{k-2})w^{k-1-j}|\zeta_{L,S}|_H^2|dw|^2.
\end{equation*}
The second term of the above equation is bounded for $|w|\geq1$ while the third one is bounded for $|w|\leq1$.
From this we check that the coefficients given by \eqref{bj-coeff} are well-defined. 

The last part of the Lemma will be proved using the usual trick on managing expansions for the function $1/(w'-w)$.
Fix $w\neq\infty$ and compute from equation \eqref{f-exp-eta}
\begin{equation}\label{expand-f(w)}\begin{aligned}
f(w)&=\intl_{S^2}\frac{|\eta|_H^2}g\,\frac{ww'}{w-w'}\,\nu(w')=\\
&=\!\!\!\!\intl_{|w'|<2|w|}\!\!\!\!\frac{|\eta|_H^2}g\,\frac{ww'}{w-w'}\,\nu(w')+\!\!\!\intl_{|w'|\geq2|w|}\!\!\!
\frac{|\eta|_H^2}g\,\frac{ww'}{w-w'}\,\nu(w')=\\
&=\!\!\!\!\intl_{|w'|<2|w|}\!\!\!\!\frac{|\eta|_H^2}g\,\frac{ww'}{w-w'}\,\nu(w') -\!\!\!\intl_{|w'|\geq2|w|}\!\!\!
\frac{|\eta|_H^2}g\left(\sum_{m=0}^\infty\frac{w^{m+1}}{(w')^m}\right)\nu(w').\\
\end{aligned}\end{equation}
The power series appearing in the above equation converges absolutely, and the respective integral can be written
\begin{equation}\label{exp-pf}\begin{aligned}
\!\!\!\intl_{|w'|\geq2|w|}&\!\!\!\frac{|\eta|_H^2}g\left(\sum_{m=0}^\infty\frac{w^{m+1}}{(w')^m}\right)\nu(w')=
\sum_{m=0}^{k-2}w^{m+1}\!\!\!\!\!\!\intl_{|w'|\geq2|w|}\!\!\!\frac{|\eta|_H^2}g(w')^{-m}\,\nu(w')\,+\\
&\qquad+\sum_{m=k-1}^\infty w^{m+1}\!\!\!\!\!\!\intl_{|w'|\geq2|w|}\!\!\!\frac{|\eta|_H^2}g(w')^{-m}\,\nu(w')=\\
=&\sum_{j=1}^{k-1}w^j\!\!\intl_{S^2}\frac{|\eta|_H^2}g(w')^{-j+1}\,\nu(w')-
\sum_{m=0}^{k-2}w^{m+1}\!\!\!\!\!\!\intl_{|w'|<2|w|}\!\!\!\frac{|\eta|_H^2}g(w')^{-m}\,\nu(w')+\\
&+\sum_{m=k-1}^\infty w^{m+1}\!\!\!\!\!\!\intl_{|w'|\geq2|w|}\!\!\!\frac{|\eta|_H^2}g(w')^{-m}\,\nu(w').
\end{aligned}\end{equation}

Let $p_f(w)$ be the polynomial given by the Lemma, and ${\cal O}=f+p_f$. Combining \eqref{expand-f(w)} and
\eqref{exp-pf} we obtain
\begin{equation}\label{v.ord.O}\begin{aligned}
{\cal O}&(w)=f(w)+p_f(w)=\!\!\!\!\intl_{|w'|<2|w|}\!\!\!\!\frac{|\eta|_H^2}g\,\frac{ww'}{w-w'}\,\nu(w')-\\
&-\sum_{j=1}^{k-1}w^j\!\!\intl_{S^2}\frac{|\eta|_H^2}g(w')^{-j+1}\,\nu(w')+
\sum_{m=0}^{k-2}w^{m+1}\!\!\!\!\!\!\intl_{|w'|<2|w|}\!\!\!\frac{|\eta|_H^2}g(w')^{-m}\,\nu(w')-\\
&-\sum_{m=k-1}^\infty w^{m+1}\!\!\!\!\!\!\intl_{|w'|\geq2|w|}\!\!\!\frac{|\eta|_H^2}g(w')^{-m}\,\nu(w')
+p_f(w).
\end{aligned}\end{equation}

The second and fifth terms of the last member of \eqref{v.ord.O} cancell out. We end up with
\begin{equation}\begin{aligned}
{\cal O}&(w)=\!\!\!\!\intl_{|w'|<2|w|}\!\!\!\!\frac{|\eta|_H^2}g\,\frac{ww'}{w-w'}\,\nu(w')+
\sum_{m=0}^{k-2}w^{m+1}\!\!\!\!\!\!\intl_{|w'|<2|w|}\!\!\!\frac{|\eta|_H^2}g(w')^{-m}\,\nu(w')-\\
&-\!\!\!\sum_{m=k-1}^\infty w^{m+1}\!\!\!\!\!\!\intl_{|w'|\geq2|w|}\!\!\!\frac{|\eta|_H^2}g(w')^{-m}\,\nu(w')
=T_1+T_2-T_3.
\end{aligned}\end{equation}

To finish the proof we will show that an estimate of the form $|T_m|\leq C|w|^k$ holds, for $m=1,2,3$. 
We can assume $|w|\leq1$ in the computations. 
As usual in this kind of argument we denote by $C(\cdot)$ a positive parameter that depends only
on the terms inside parenthesis. Different occurrences of $C$ may mean different ``constants''.

Replacing $w$ by $w'$ in equation \eqref{eta-Hnorm} we get
\begin{equation}\label{estim-w<1}
  \frac{|\eta|_H^2}g=(w')^{k-2}(\ov g\,(\ov w')^{k-2})|\zeta_{L,S}|_H^2|dw'|^2=(w')^{k-2}\,M_H(w'),
\end{equation}
where $M_H(w')$ remains bounded if $|w'|\leq2$. Thus
\begin{equation}\label{est-T1}\begin{aligned}
|T_1|&\leq\!\!\!\!\intl_{|w'|<2|w|}\!\!\!\!\frac{|\eta|_H^2}{|g|}\,\left|\frac{ww'}{w-w'}\right|\,\nu(w')\\
&\leq\!\!\!\!\intl_{|w'|<2|w|}\!\!\!\! |w'|^{k-2}\,|M_H(w')|\frac{|ww'|}{|w-w'|}\,\nu(w')\\
&<|w|^k2^{k-1}\|M_H\|_{L^\infty(|w'|<2)}\intl_{|w'|<2|w|}\!\!\!\!\frac{\nu(w')}{|w-w'|}.
\end{aligned}\end{equation}
An easy estimate shows that 
\begin{equation}\label{est-1/w-w'}
\intl_{|w'|<2|w|}\!\!\!\!\frac{\nu(w')}{|w-w'|}<C\qquad\text{if}\ |w|\leq1.
\end{equation}
Hence from \eqref{est-T1} and \eqref{est-1/w-w'} we obtain
\begin{equation}\label{T1-estimado}
  |T_1|\leq C(H)\,|w|^k\qquad\text{if}\ |w|\leq1.
\end{equation}

The estimate for $T_2$ follows a similar line to $T_1$:
\begin{equation}\label{est-T2-parc}\begin{aligned}
|T_2|&=\left|\sum_{m=0}^{k-2}w^{m+1}\!\!\!\!\!\!\intl_{|w'|<2|w|}\!\!\!\frac{|\eta|_H^2}g(w')^{-m}\,\nu(w')\right|\\
&\leq\sum_{m=0}^{k-2}|w|^{m+1}\!\!\!\!\!\!\intl_{|w'|<2|w|}\!\!\! |M_H(w')|\,|w'|^{k-2}|w'|^{-m}\,\nu(w')\\
&\leq\sum_{m=0}^{k-2}|w|^{m+1}|2w|^{k-2-m}\|M_H\|_{L^\infty(|w'|<2)}\!\!\!\intl_{|w'|<2|w|}\!\!\! \nu(w')\\
&\leq C(H)|w|^{k-1}\!\!\!\intl_{|w'|<2|w|}\!\!\! \nu(w').
\end{aligned}\end{equation}
The integral in the last member of \eqref{est-T2-parc} is the area of a geodesic disc of radius $R$. This disk
has area smaller than its image under the conformal mapping $w':S^2-\{S\}\to\R^2$, so
$\!\!\!\intl_{|w'|<2|w|}\!\!\! \nu(w')<C\,|w|^2$. We get
\begin{equation}\label{T2-estimado}
  |T_2|\leq C(H)|w|^{k+1}.
\end{equation}

Finally we estimate $T_3$.
\begin{equation}\begin{aligned}
 T_3&=\!\!\!\sum_{m=k-1}^\infty w^{m+1}\!\!\!\!\!\!\intl_{|w'|\geq2|w|}\!\!\!\!\!\!\frac{|\eta|_H^2}g(w')^{-m}\,\nu(w')=
w\!\!\!\!\!\!\!\intl_{|w'|\geq2|w|}\!\!\!\!\!\!\!\frac{|\eta|_H^2}g\!\!\sum_{m=k-1}^\infty\!\!\left(\frac w{w'}\right)^m
\nu(w')=\\
&\quad=w\!\!\!\!\!\!\!\intl_{|w'|\geq2|w|}\!\!\!\!\!\frac{|\eta|_H^2}{g(w')^{k-1}}\,w^{k-1}\!\!\sum_{m=0}^\infty
\left(\frac w{w'}\right)^m\,\nu(w').
\end{aligned}\end{equation}
Therefore,
\begin{equation}\label{est-T3}
|T_3|\leq 2|w|^k\!\!\!\!\!\!\!\intl_{|w'|\geq2|w|}\!\!\!\!\!\frac{|\eta|_H^2}{|g||w'|^{k-1}}\,\nu(w')<
C\,|w|^k\!\intl_{S^2}\frac{|\eta|_H^2}{|g||w'|^{k-1}}\,\nu(w').
\end{equation}
The integral in $S^2$ can be split into integrals in the north and south hemispheres. The first of them satisfies
\begin{equation}\label{north-int-est}
\!\!\!\!\!\!\!\intl_{|w'|\leq1}\!\!\!\!\!\frac{|\eta|_H^2}{|g||w'|^{k-1}}\,\nu(w')\,\,<
\!\!\intl_{|w'|\leq1}\!\!\!\!\!\frac{|M_H(w')|}{|w'|}\,\nu(w')\leq C\,\|M_H\|_{L^\infty(|w'|\leq1)}.
\end{equation}
The south hemisphere integral is estimated as
\begin{equation}\label{south-int-est}\begin{aligned}
 \!\!\!\!\!\!\!\intl_{|w'|>1}\!\!\!\!\!&\frac{|\eta|_H^2}{|g||w'|^{k-1}}\,\nu(w')\,\,<
\!\!\!\intl_{|w'|>1}\!\!\!\!\!\frac{|\eta|_H^2}{|g|}\,\nu(w')\,=
\!\!\!\intl_{|w'|>1}\!\!\!\!\!|g||\zeta_L|_H^2|dz|^2\,\nu(w')\leq\\
&\leq C\,\|g\|_{L^\infty(|w'|>1)}\|\zeta_L\|_{L^\infty_H(|w'|>1)}^2\|dz\|_{L^\infty(|w'|>1)}^2.
\end{aligned}\end{equation}
From \eqref{est-T3}, \eqref{north-int-est} and \eqref{south-int-est} we obtain
\begin{equation}\label{T3-estimado}
  |T_3|\leq C(H)|w|^k.
\end{equation}

Altogether inequalities \eqref{T1-estimado}, \eqref{T2-estimado} and \eqref{T3-estimado} imply 
$|{\cal O}(w)|\leq C|w|^k$. This completes with the Lemma's proof.
\end{proof}

\begin{proof}[Proof of Lemma \ref{meromorphic-on-E}]
Let's first assume \emph{there is} some meromorphic section $\psi$ on $E$, not contained in $L_1$.
Then $\psi=(\psi_1,\psi_2)$ and up to its set of poles it must satisfy
\begin{equation}\label{psi-meromorphic}
\left\{\begin{aligned} & D_1''(\psi_1)+\eta(\psi_2)=0 \\ & D_2''(\psi_2)=0.\end{aligned}\right.
\end{equation}
Therefore $\psi_2$ is meromorphic on $L_2$. Multiplying $\psi$ by a suitable meromorphic function
we can assume that $\psi_2=\zeta_2$. Similarly, we can write $\psi_1=f.\zeta_1$ 
for some function $f$ smooth outside of the singular set of $\psi$.
Recalling that in coordinates we have 
\begin{equation*}
\eta=\ov g\cdot\frac{\zeta_1}{|\zeta_1|_H^2}\cdot\zeta_2^{\ast H}\cdot d\ov z,
\end{equation*}
the first of equations \eqref{psi-meromorphic} holds (using $\psi_2=\zeta_2$) if and only if
\begin{equation*}
 \pb f+\ov g\,\frac{|\zeta_2|_H^2}{|\zeta_1|_H^2}d\ov z=0,
\end{equation*}
or equivalently,
\begin{equation*}
\pb_z f=-\ov g\,|\zeta_L|_H^2.
\end{equation*}

Clearly the above steps can be reversed, and if we start off at the solution $f$ for \eqref{delzbar-f=},
given by Lemma \ref{sol_delbar-eq}, we construct a meromorphic section $\psi$ satisfying the conditions
of Lemma \ref{meromorphic-on-E}.

Now assume $J\subset E$ is a line subbundle different from $L_1$. We pick a meromorphic section
$\tilde\psi$ spanning $J$, and because of the above argument, can assume $\tilde\psi=(\tilde f\,\zeta_1,\zeta_2)$.
Further $\tilde f$ satisfies equation \eqref{delzbar-f=} in all but finite many points.
We conclude that $\tilde f-f=h$ is meromorphic and $\tilde\psi$ has the form $\tilde\psi=((f+h)\zeta_1,\zeta_2)$.
\end{proof}

Now let $J\subset E$ be a holomorphic line subbundle, and $\tilde\psi$ be a meromorphic section
spanning $J$, in the form given by Lemma \ref{meromorphic-on-E}, with function $f$ vanishing at $N$,
as given by Lemma \ref{sol_delbar-eq}. At any $x\in S^2-\{N\}$ 
$\zeta_2$ is regular, thus $x$ cannot be a zero of $\tilde\psi$. And $x$ is a pole of
$\tilde\psi$ if and only if $x$ is a pole of $h$ of the same order. Hence 
\begin{equation}\label{ord-psitil-1}
\ord_x(\tilde\psi)=\min\{0,\ord_x(h)\}\qquad\text{for}\ x\neq N.
\end{equation}

On a vicinity of the north pole we write 
\begin{equation}\begin{aligned}
  \tilde\psi(w)&=((f(w)+h(w))\zeta_1(w),\zeta_2(w))\\
&=(({\cal O}(w)-p_f(w)+h(w))w^{\Deg(L_1)}\zeta_{L_1,S},w^{\Deg(L_2)}\zeta_{L_2,S})\\
&=w^{\Deg(L_1)}\,(({\cal O}(w)-p_f(w)+h(w))\zeta_{L_1,S},w^k\zeta_{L_2,S}).
\end{aligned}\end{equation}
Observe that $\ord_N(\tilde\psi)=m$ if and only if $m$ is the only integer such that $w^{-m}\tilde\psi$ is
a regular holomorphic section around $w=0$. Because $\frac{{\cal O}(w)}{w^k}$ (for $w\neq0$) is bounded, a quick
study of the cases $\ord_N(h-p_f)<k$ and $\ord_N(h-p_f)\geq k$ (take this order to be infinite if $h-p_f$ is null)
leads to 
\begin{equation}\label{ord-psitil-2}
  \ord_N(\tilde\psi)=\Deg(L_1)+\min\{\ord_N(h-p_f),k\}.
\end{equation}
Let $s^-$ denote the number of poles (accounting for multiplicity) of $h$ in $S^2-\{N\}$. 
From \eqref{ord-psitil-1} and \eqref{ord-psitil-2} we get
\begin{equation}\label{deg-psitil-form}
 \Deg(\tilde\psi)=\sum_{x\in S^2}\ord_x(\tilde\psi)=\Deg(L_1)+\min\{\ord_N(h-p_f),k\}-s^-.
\end{equation}
Our aim is to compute $\Div E$, which is the maximum among the degrees of $\tilde\psi$  for all such meromorphic
sections. Thus we need to find an appropriate meromorphic $h$ that maximizes the right-hand-side of 
\eqref{deg-psitil-form}. Since $p_f(N)=0$ we should choose $h$ so that $h(N)=0$, otherwise we would 
have $\ord_N(h)\leq0$, so $\ord_N(h-p_f)=\ord_N(h)$ and $\Deg(\tilde\psi)=\Deg(L_1)+\ord_N(h)-s^-\leq\Deg(L_1)$. 
In particular we can write, without loss of generality,
\begin{equation}
 h(w)=\frac{y(w)}{1-v(w)},
\end{equation}
where $y(w)$ and $v(w)$ are polynomials, $y(0)=0=v(0)$ and $y$ and $1-v$ have no common zeros.
The number of poles $s^-$ of $h$ equals the maximum degree among the polynomials $y(w)$ and $1-v(w)$,
hence to allow $\Deg(\tilde\psi)>\Deg(L_1)$ we can assume both degrees to be less than $k$.
\begin{lem}\label{calc-deg}
Follow the above notation and conditions for $y(w)$, $v(w)$  and $s^-$, and for any polynomial in $w$
denote by a subindex $j$ the coefficient of $w^j$ in it. Let $\{b_j\}$ be the coefficients given by 
\eqref{bj-coeff}. Consider the system of equations
\begin{equation}\label{equat-ord}
\left\{\begin{aligned}
b_1&=y_1 \\ b_2&=y_2+(yv)_2 \\ b_3&=y_3+(yv)_3+(yv^2)_3 \\ & \vdots   \\
b_{k-1}&=y_{k-1}+(yv)_{k-1}+(yv^2)_{k-1}+\ldots+(yv^{k-2})_{k-1}\end{aligned}\right.
\end{equation}
Let $j^\ast\leq k$ be the maximum integer such that all equations in system \eqref{equat-ord}
with index $j<j^\ast$  are satisfied. Then 
\begin{equation}\label{deg-psitilde}
{\rm deg}(\tilde\psi)={\rm deg}(L_1)+j^\ast-s^-.
\end{equation}
\end{lem} 
\begin{proof}
We only need to show that $j^\ast=\min\{\ord_N(h-p_f),k\}$ and use equation \eqref{deg-psitil-form}.
Because $h$ is holomorphic at $w=0$ we can write $h(w)=p_h(w)+{\cal O}_h(w)$ where $p_h(w)$ is
a polynomial of degree lower than $k$ and ${\cal O}_h=h-p_h$ has order greater than $k-1$ in $w=0$.
Then $\min\{\ord_N(h-p_f),k\}=\min\{\ord_N(p_h-p_f),k\}$. Expanding $h$ in the polynomials $y$ and $v$ 
close to $w=0$ we get
\begin{equation}\label{h-expand}
  h(w)=\frac{y(w)}{1-v(w)}=\sum_{m=0}^\infty y(w)v(w)^m=p_h(w)+{\cal O}_h(w).
\end{equation}
For any order $1\leq j\leq k-1$ the only summands in the third member of \eqref{h-expand} that 
add up to the $j$-th coefficient of $p_h$ are those $y\,v^m$ with $m<j$. Hence,
\begin{equation}
  {p_h}_j=y_j+(yv)_j+(yv^2)_j+\ldots+(yv^{j-1})_j.
\end{equation}
Therefore the $j$-th equation of system \eqref{equat-ord} is nothing but a statement of
equality between the $j$-th coefficients of $p_f$ and $p_h$. If $j^\ast<k$ then
all such equations for $j<j^\ast$ are satisfied but the equation for $j=j^\ast$ is not, thus the 
first non-vanishing coefficient of $p_h-p_f$ is $(p_h-p_f)_{j^\ast}$. 
If $j^\ast=k$ then $p_h-p_f\equiv0$. In both cases one has $\min\{\ord_N(p_h-p_f),k\}=j^\ast$.
\end{proof}

The practical application of Lemma \ref{calc-deg} will be shown on Section \ref{s2}.
For now it is interesting to notice that $\Div[\eta]$ will appear as the maximum right-hand-side value of
equation \eqref{deg-psitilde}. This value depends on the parameters $j^\ast$, $s^-$ and ultimately,
on the coefficients $b_j$ for $1\leq j\leq k-1$. However, the latter seem to depend upon the metric $H$, 
besides the very cohomology class $[\eta]$, after equation \eqref{bj-coeff}. Amazingly, it turns out
that $\{b_j\}$ \emph{do not depend upon the metric}, as the next result states.
\begin{lem}
Let $\beta=\{z^{j-1}\zeta_Ldz\}_{1\leq j<k}$ be a basis of $\hh$, and let $\beta^\ast$ be the
dual cannonical basis of $\ha$. Then for a given $[\eta]\in\ha$ the coefficients $\{b_j\}$ obtained
from formula \eqref{bj-coeff} using any metric are the coordinates of $[\eta]$ in $\beta^\ast$.
\end{lem}
\begin{proof}
Fix a metric $H$ and let $\eta$ be the $H$-antiholomorphic representative for $[\eta]$. 
Set $\phi=\eta^{\ast H}$, thus $\phi=g\,\zeta_Ldz$ for some polynomial $g$. Then at each $x\in S^2$,
\begin{equation}
 |\eta|_H^2=i\Lambda(\eta^{\ast H}\wedge\eta)=i\Lambda(\phi\wedge\eta).
\end{equation}
Equations \eqref{bj-coeff} turn into
\begin{equation}\label{hjdsh}
  b_j=\int_{S^2}\frac ig(\phi\wedge\eta)\,z^{j-1}=\int_{S^2}i(z^{j-1}\zeta_Ldz\wedge\eta)=
\cpl [z^{j-1}\zeta_Ldz],[\eta]\cpr,
\end{equation}
hence $b_j$ is the coupling of $[\eta]$ with the $j$-th vector of the basis $\beta$.
\end{proof}

\section{Some Conformal Curvatures on $S^2$}\label{s2}

In this section we use the previous theory to show existence of metrics pointwise conformal to 
the standard metric on $S^2$ for some non-negative curvatures with zeros. 

\subsection{Projectivized Cohomology as a Parameter Space}

We set one more equivalence to simplify our analysis. Let $\alpha$ be a non-zero complex
constant. If $[\eta]\in\ha$ is non-zero, $\eta$ represents $[\eta]$, then $(u,\xi)$ solves 
\eqref{system-eta} if and only if $(u+\ln|\alpha|,\alpha\xi)$ solves \eqref{system-eta} after
replacing $\eta$ by $\alpha\eta$. Solutions for classes that are multiple of each other differ
by a constant. The case $[\eta]=0$ is of no interest for equation $\Delta u-\lb=0$ has no solution 
at all if $\lb\neq0$. This motivates us to work on the projectivization
\begin{equation}
  \Pha=\frac{\ha-\{0\}}{[\eta]\sim\alpha[\eta]}\simeq\mathbb{CP}^{k-2}.
\end{equation}
We similarly define $\Phh$ as the projectivization of $\hh$ and the natural home for
function parameters for equation \eqref{system-phi}. For a metric $H$ the function 
given by \eqref{ha-map-hh} is homogeneous and passes to a diffeomorphism $\ast H:\Phh\to\Pha$.
To avoid cumbersome notation we will denote the projective class of some $[\eta]\in\ha$ ($[\phi]\in\hh$) 
by the same symbol $[\eta]\in\Pha$ ($[\phi]\in\Phh$).
Though we must take care of the scaling when consider equations \eqref{system-eta} and \eqref{system-phi}.
Hence we denote by $u=u([\eta],\lb)$ the \emph{zero mean value component} of a solution for 
\eqref{system-eta}. For the given projective $[\eta]$ we choose any smooth section representative 
$\eta\in[\eta]$: the solution of \eqref{system-eta} is given by $u+C$ for a uniquely defined
real constant $C$ (as long as $\lb$ is in the existence range). This approach seems good to us
because allows the definition of the function $u$ given by Lemma \ref{exist-sol-eta} directly in
$\Pha$ and avoids the necessity of a normalization condition on $\eta$.

Let $m\geq1$. Define
\begin{equation}
  \Pha_m=\{[\eta]\in\Pha\,|\,\Div[\eta]\geq\Deg(L_2)-m\}.
\end{equation}
The interest on the sets $\Pha_m$ stands for a neat paraphrase of Lemma \ref{existence-terms-diveta}:
\begin{cor}\label{eta-range-exist}
If $m>q$ then $\Pha_m\supset\Pha_q$. For $[\eta]\in\Pha$ and $m\leq{\rm deg}(L)$, $m\in\Z$, 
it holds $[\eta]\in\Pha_m-\Pha_{m-1}$ if and only if the range of values of $\lb\in(0,4\pi\,{\rm deg}(L))$ 
for which there are solutions of \eqref{system-eta} is $(0,4\pi m)$. 
\end{cor}
From Lemma \ref{div-eta<=degL2} and the argument in the proof of Lemma \ref{existence-terms-diveta} we get,
for any $[\eta]$, $\Deg(L_2)-1\geq\Div[\eta]\geq \Deg(L_2)-\lfloor\frac k2\rfloor$, and thus the decreasing sequence
\begin{equation}
  \Pha=\Pha_{\lfloor\frac k2\rfloor}\supset\Pha_{\lfloor\frac k2\rfloor-1}\supset\Pha_{\lfloor\frac k2\rfloor-2}\supset\cdots
\supset\Pha_2\supset\Pha_1.
\end{equation}
The notation is suggestive in the sense that we conjecture all $\Pha_m$ are copies of $\mathbb{CP}^r$, for
different dimensions $r$, inside $\Pha\simeq\mathbb{CP}^{k-2}$. We have not been able to prove it so far, but
only for the ending terms of the sequence.
\begin{lem}\label{P_1-low-range}
There is a complex embedding $\mathbb{CP}^1\to\Pha$ which is a diffeomorphism onto $\Pha_1$. 
For any $[\eta]\in\Pha_1$ the divisor of the class $[\eta]^{\ast H_0}$ is $(k-2)x$ for some $x\in S^2$.
\end{lem}
\begin{proof}
Let $[\eta]\in\Pha_1$, thus $\Div[\eta]=\Deg(L_2)-1=\Deg(L_1)+k-1$. Following Lemma \ref{calc-deg} and
equation \eqref{deg-psitilde} for a section $\tilde\psi$ with maximal degree we find that $j^\ast-s^-=k-1$.
Because of the bounds $1\leq j^\ast\leq k$ and $s^-\geq 0$ we get $s^-\leq1$. The polynomials 
$y(w)$ and $v(w)$ are linear or null, and system \eqref{equat-ord} turns into
\begin{equation}\label{equat-ord-P_1}
\left\{\begin{aligned}
b_1&=y_1 \\ b_2&=y_1v_1 \\ b_3&=y_1v_1^2\\ & \vdots   \\
b_{k-1}&=y_1v_1^{k-2}\end{aligned}\right.
\end{equation}
In case $s^-=0$ and $j^\ast=k-1$ the meromorphic function $h$ of the Lemma is identically zero, so $y_1=0$.
We get $b_j=0$ for $1\leq j\leq k-2$ and $b_{k-1}\neq0$. Otherwise, $s^-=1$ and $j^\ast=k$. Thus
$y_1\neq0$ and all equations in \eqref{equat-ord-P_1} are satisfied. With this characterization it is easy 
to see that the function
\begin{equation}
  \Psi[b_1:b_2]=\left\{\begin{aligned} & \left[1:\frac{b_2}{b_1}:\frac{b_2^2}{b_1^2}:\cdots:
\frac{b_2^{k-2}}{b_1^{k-2}}\right]\qquad \text{if}\ b_1\neq0\\
   & \left[\frac{b_1^{k-2}}{b_2^{k-2}}:\frac{b_1^{k-3}}{b_2^{k-3}}:\cdots:\frac{b_1}{b_2}:1\right]\qquad 
\text{if}\ b_2\neq0
\end{aligned}\right.
\end{equation}
is a diffeomorphism from $\mathbb{CP}^1$ onto the homogeneous coordinates of the classes $[\eta]$ with
$\Div[\eta]=\Deg(L_2)-1$.

Now we look for possibilities for the divisor of $[\phi]=[\eta]^{\ast H_0}$. First consider the case
$\phi=\zeta_Ldz$ (hence $g$ is a constant). In the computation of $b_j$ in formula \eqref{bj-coeff} we can replace
$|\eta|_{H_0}^2$ by $|\phi|_{H_0}^2$. Due to the rotational symmetry for the metric $H_0$ in \eqref{H0.def} 
and of the holomorphic coordinate, the integrals \eqref{bj-coeff} vanish for $j>1$ and is non-zero in $j=1$,
therefore the coefficients associated to $[\zeta_Ldz]^{\ast H_0}$ are $b_1\neq0$ and $b_j=0$, $2\leq j\leq k-1$.
We conclude that $[\zeta_Ldz]^{\ast H_0}=\Psi[1:0]\in\Pha_1$. 

In general, let $\phi=(z-a)^{k-2}\zeta_Ldz$, where $a=z(x_0)$ for some $x_0\in S^2$. A not so short 
analytic argument to show that $[\phi]^{\ast H_0}$ is in $\Pha_1$ is simply to compute $b_j$ 
with formula \eqref{bj-coeff} and showing those are in geometric progression. A more direct geometric approach, 
though, is noticing that the coefficients given by \eqref{bj-coeff} depend on the 
basis $\{z^j\zeta_Ldz\}_{0\leq j\leq k-2}$ of $\hh$. Change this basis to $\{\tilde z^j\phi\}_{0\leq j\leq k-2}$
where $\tilde z=z_{x_0}$ is a stereographic coordinate satisfying $\tilde z(-x_0)=0$, and use 
$\tilde w=1/\tilde z$ to replace $w$ in the integrals \eqref{bj-coeff}. 
Clearly the whole construction of Lemmas \ref{meromorphic-on-E} and
\ref{sol_delbar-eq} does not depend on the fact that $N=(0,0,1)$, or rather, on the coordinate chart
used. In the new charts given by $\tilde z$ (or $\tilde w$) and $\phi$, the argument follows like in the previous
paragraph, so $[\phi]^{\ast H_0}\in\Pha_1$ in this case also. This shows that all classes $[\phi]\in\Phh$
with a zero of order $k-2$ are the images of classes in $\Pha_1$ under $\ast H_0$. The conclusion then follows since
both of the set of those classes, as well as $\Pha_1$, are diffeomorphic to $\mathbb{CP}^1$, and $\ast H_0$
is a diffeomorphism between them.
\end{proof} 

\subsection{The isometry group of $S^2$}

Let $\vf:S^2\to S^2$ be an isometry. Take points $x,y\in S^2$ with $\vf(x)=y$. Choose stereographic
coordinates $z,v$ around $x$ and $y$, respectively, such that $z(x)=v(y)=0$. Since $\vf$ is conformal
and is an isometry it is not hard to see that $v=\vf(z)=b z$ for a unitary complex $b$, if $\vf$
preserves orientation, and $v=b \ov z$, if $\vf$ reverses orientation. Then,
for $h$ a complex-valued function on $S^2$ we set
for any $x\in S^2$
\begin{equation}
\vfs h(x)=\left\{\begin{aligned}& h(\vf(x))\qquad\text{if}\ \vf\ \text{is orientation preserving}\\
   & \ov{h(\vf(x))}\qquad\text{if}\ \vf\ \text{is orientation reversing}\end{aligned}\right.
\end{equation}
The conjugation in the second case above aims to preserve holomorphicity: $h$ is holomorphic in some
open set $U\subset S^2$ if and only if $\vfs h$ is holomorphic in $\vf^{-1}(U)$. 
This definition is naturally extended to a complex-valued differential form $\om$: writing locally 
$\om=h\,\mu$ for $h$ a function and $\mu$ a real-valued form we set $\vfs\om=\vfs h\vfs\mu$, where $\vfs\mu$ is 
the usual pull-back of forms.

We must define a similar notion for classes in $\Phh$ and $\Pha$. This is not that simple because there 
is no cannonical identification between the fibers $L_x$ and $L_{\vf(x)}$, for $x$ in $S^2$.
We do that by first defining the pull-back of divisors. If ${\cal D}=\sum_j a_jx_j$ we set 
$\vfs{\cal D}=\sum_ja_j\vf^{-1}(x_j)$. 

Now fix some holomorphic $\zeta$ in $L$ whose divisor is ${\cal D}$ and set $\vfs\zeta$ as some 
non-trivial holomorphic section with divisor $\vfs{\cal D}$. 
If $\psi$ is an arbitrary smooth section in $\Om^{p,q}(L)$ then $\psi=\om\otimes\zeta=\om\zeta$ for some form $\om$ 
smooth away of the singular set of $\zeta$. Define
\begin{equation}
\vfs\psi=\vfs\om\vfs\zeta.
\end{equation}
\begin{lem}\label{prop.vfs.em-holom}
Let $H=H_u$ be a metric. Then:\\
(i) The operator $D''$ commutes with $\vfs$. In particular, $\psi\in\Om^{p,q}(L)$ is meromorphic with 
divisor ${\cal D}$ if and only if $\vfs\psi$ is meromorphic with divisor $\vfs{\cal D}$.
(ii) There is a constant $c>0$ such that for any $\psi,\chi\in\Om^{p,q}(L)$ it holds
\begin{equation}\label{vfs-metric}
\vfs\langle\psi\wedge\chi\rangle_{H_u}=c\langle\vfs\psi\wedge\vfs\chi\rangle_{H_{\vfs u}}.
\end{equation}

\noindent
(iii) For any section-valued form $\psi$  one has $\vfs(D_u'\psi)=D_{\vfs u}'(\vfs\psi)$. In particular,
$\vfs$ commutes with $D_{H_0}'$.
\end{lem}
\begin{proof}
(i) 
Let $\psi$ be a smooth $(p,q)$-section, then $\psi=\om\zeta$. Thus
\begin{equation}\begin{aligned}
D''(\vfs\psi)&=D''(\vfs\om\vfs\zeta)=\pb(\vfs\om)\vfs\zeta=\vfs\pb\om\vfs\zeta=\\
    &=\vfs(\pb\om\zeta)=\vfs D''\psi,
\end{aligned}\end{equation}
since both of $\zeta$ and $\vfs\zeta$ are holomorphic and $\pb$ commutes with $\vfs$ by a property of
the pull-back on forms. Therefore $\vfs$ takes meromorphic sections to meromorphic sections.
Let $\psi=\om\zeta$ be meromorphic. Then $\om$ is meromorphic. Tensoring meromorphic 
sections adds up their divisors, hence
\begin{equation}\begin{aligned}
{\cal D}(\vfs\psi)&={\cal D}(\vfs\om)+{\cal D}(\vfs\zeta)=\vfs{\cal D}(\om)+\vfs{\cal D}(\zeta)\\
&=\vfs{\cal D}(\om\zeta)=\vfs{\cal D}(\psi).
\end{aligned}\end{equation}

\noindent
(ii)
Consider first that $\psi=\chi=\zeta$ and $u=0$, so $H=H_0$. Equality \eqref{vfs-metric} turns into 
$\vfs|\zeta|_{H_0}^2=c|\vfs\zeta|_{H_0}^2$. Notice that both members of this equation 
are functions with singularities in the same points, namely the divisor set of $\vfs\zeta$. 
We claim that the function $f=\ln\left(\frac{\vfs|\zeta|_{H_0}^2}{|\vfs\zeta|_{H_0}^2}\right)$ is
actually smooth in $S^2$. For any $x$ in this singular set, let $y=\vf(x)$, and take $z$, $v$ coordinates
centered at $x$, $y$, and such that $v(z)=\vf(z)=z$ (assume without loss of generality $\vf$ is 
orientation preserving). Then
\begin{equation}
\zeta(v)=v^m\tilde\zeta_y(v),\quad\vfs\zeta(z)=z^m\tilde\zeta_x(z)
\end{equation}
where $\tilde\zeta_y,\tilde\zeta_x$ are regular around $v=0$ and $z=0$, respectively, and
\begin{equation}
f(z)=\ln\left(\frac{|\zeta(v(z))|_{H_0}^2}{|\vfs\zeta(z)|_{H_0}^2}\right)=
\ln\left(\frac{|\tilde\zeta_y(v(z))|_{H_0}^2}{|\tilde\zeta_x(z)|_{H_0}^2}\right)
\end{equation}
is clearly smooth at $x$. Thus $f\in C^\infty(S^2)$.

Now recall Lemma \ref{curvature-merom} and compute
\begin{equation}\begin{aligned}
\Delta f&=\Delta\left(\ln(\vfs|\zeta|_{H_0}^2)-\ln|\vfs\zeta|_{H_0}^2\right)\\
&=\vfs(\Delta\ln|\zeta|_{H_0}^2)-\Delta\ln|\vfs\zeta|_{H_0}^2\\
&=\vfs(-4\pi\Deg(L))+4\pi\Deg(L)=0.
\end{aligned}\end{equation}
In the above we used that $\zeta$, $\vfs\zeta$ are holomorphic and that $\vfs$ commutes with $\Delta$.
In particular we obtain that $f$ is harmonic in the whole sphere, so $f$ is constant, and 
equation \eqref{vfs-metric} follows immediately in this particular case for an appropriate $c>0$. 
Now for an arbitrary $u$,
\begin{equation}
\vfs|\zeta|_{H_u}^2=\vfs(|\zeta|_{H_0}^2\,e^{2u})=c|\vfs\zeta|_{H_0}^2\,e^{2\vfs u}=
c|\vfs\zeta|_{H_{\vfs u}}^2.
\end{equation}
The general case is a consequence of this one once we write the 
section-valued forms $\psi,\chi$ in components with $\zeta$.

\noindent
(iii) 
Again, the general case will follow as routine if we prove it for the very case $\psi=\zeta$.
The section $D_u'\zeta$ can be managed implicitly in the equation 
$\pa|\zeta|_{H_u}^2=\langle D_u'\zeta,\zeta\rangle_{H_u}$. Applying $\vfs$ to it and using part (ii) we derive
\begin{equation}\label{vfs-del}
\vfs\pa|\zeta|_{H_u}^2=c\langle\vfs D_u'\zeta,\vfs\zeta\rangle_{H_{\vfs u}}.
\end{equation}
Interchanging $\vfs$ and $\pa$ in the above equation yields
\begin{equation}\label{del-vfs}
\pa\vfs|\zeta|_{H_u}^2=\pa\left(c|\vfs\zeta|_{H_{\vfs u}}^2\right)=
c\langle D_{\vfs u}'\vfs\zeta,\vfs\zeta\rangle_{H_{\vfs u}},
\end{equation}
and since $\pa$ and $\vfs$ commute, the last members of equations \eqref{vfs-del} and
\eqref{del-vfs} are equal. It follows $\vfs D_u'\zeta=D_{\vfs u}'\vfs\zeta$. This finishes with
the Lemma's proof.
\end{proof}

It becomes suitable to define the pull-back of a metric: for $H=H_0e^{2u}$ we set
$\vfs H_u=H_{\vfs u}$. The definition of $\vfs$ on sections of the bundle $L^\ast$ is now very natural. 
For a section $\xi$ on $L^\ast$ we define $\vfs\xi=(\vfs(\xi^{\ast H}))^{\ast\vfs H}$.
Clearly one has to show invariance from the metric's choice. 
\begin{equation}\begin{aligned}
\vfs\xi&=(\vfs(\xi^{\ast H_u}))^{\ast\vfs H_u}=(\vfs(\xi^{\ast H_0})(\vfs e^{-2u}))^{\ast H_{\vfs u}}\\
    &=(\vfs(\xi^{\ast H_0}))^{\ast H_0}(e^{-2\vfs u})(e^{2\vfs u})=(\vfs(\xi^{\ast H_0}))^{\ast H_0}.
\end{aligned}\end{equation}
The proof of the next Lemma will be skipped.
\begin{lem}\label{prop.vfs.em-anti-holom}
(i) For any $\xi\in\Om^{p,q}(L^\ast)$ it holds $D''\vfs\xi=\vfs(D''\xi)$. In particular, $\vfs$
descends to the cohomology $\ha$.\\
(ii) If $\phi$, $\eta$ are sections in $\Om^{1,0}(L)$, $\Om^{0,1}(L^\ast)$, respectively, then 
$\vfs(\phi\wedge\eta)=c(\vfs\phi\wedge\vfs\eta)$, where $c$ is the same constant as in Lemma 
\ref{prop.vfs.em-holom} part (ii). In particular $\cpl\phi,\eta\cpr=c\cpl\vfs\phi,\vfs\eta\cpr$.\\
(iii) For a section $\eta\in\Om^{0,1}(L^\ast)$ it holds $\vfs(D_u'\eta)=D_{\vfs u}'(\vfs\eta)$.
\end{lem}

The whole construction of the $\vfs$ pull-back started with a particular holomorphic section of $L$.
Since it is $\C$-linear and up to a constant factor, holomorphic sections are defined by their divisors,
we conclude $\vfs$ induces pull-back in a unique way on $\Phh$ and $\Pha$. Equivalently, there is
a right action of the isometry group ${\rm Iso}(S^2)$ on the manifolds $\Phh$ and $\Pha$.

The operation of the $H$-dual given by \eqref{ha-map-hh} is defined in the projective
cohomology. In view of Corollary \ref{eta-range-exist} and the definition of $u([\eta],\lb)$ we can set
the smooth map
\begin{equation}\begin{aligned}
{\cal F}:\!\!\!\!\bigcup_{\lfloor\frac k2\rfloor\geq m\geq 1}&\!\!\!\! (\Pha_m-\Pha_{m-1})\times(0,4\pi m)\to \Phh,\\
{\cal F}([\eta],\lb)&={\cal F}_\lb[\eta]=[\eta]^{\ast H_{u([\eta],\lb)}}.
\end{aligned}\end{equation}

The function ${\cal F}$ behaves well respect to the isometries of $S^2$.
\begin{lem}\label{F-lb-commutes-iso}
Let $\vf$ be an isometry. Fix $([\eta],\lb)$ in the domain of ${\cal F}$. 
Then $\vfs{\cal F}_\lb[\eta]={\cal F}_\lb(\vfs[\eta])$.
\end{lem}
\begin{proof}
Recall that $u([\eta],\lb)$ designates the zero mean-value component of the actual solution
of the first of equations \eqref{system-eta}. Hence, writing for simplicity $u=u([\eta],\lb)$
and assuming $\eta$ is an arbitrary representative of $[\eta]\in\Pha$, we have
\begin{equation*}
\Delta u+2|\eta+D''\xi|_0^2e^{-2(u+r)}-\lb=0
\end{equation*}
for some $r\in\R$. On the other hand, applying $\vfs$ to this equation yields
\begin{equation}\begin{aligned}
0&=\vfs(\Delta u+2|\eta+D''\xi|_0^2e^{-2(u+r)}-\lb)\\
 &=\Delta\vfs u+2ce^{-2r}|\vfs(\eta+D''\xi)|_0^2e^{-2\vfs u}-\lb.
\end{aligned}\end{equation}
The second term of the third member above is justified by Lemma \ref{prop.vfs.em-holom} part (ii) together
with the observation that $\vfs$ and the contraction operator $i\Lb$ commute: 
\begin{equation}\begin{aligned}
\vfs|\eta+D''\xi|_{H_u}^2&=\vfs|\phi|_{H_u}^2=c\,i\Lb\langle\vfs\phi\wedge\vfs\phi\rangle_{H_{\vfs u}}=\\
&=c|\vfs\phi|_{H_{\vfs u}}^2=c|\vfs(\eta+D''\xi)|_{H_{\vfs u}}^2,
\end{aligned}\end{equation}
where we write $\phi=(\eta+D''\xi)^{\ast H_u}$.

Applying $\vfs$ to the second equation in \eqref{system-eta} clearly (re)states that $\vfs(\eta+D''\xi)$
is $\vfs H_u$-antiholomorphic, thanks to Lemma \eqref{prop.vfs.em-anti-holom} part (iii). This shows us
that $\vfs u+r-\frac12\ln(c)$ is a solution to \eqref{system-eta} with $\vfs(\eta+D''\xi)$ in the place of $\eta+D''\xi$. 
Since the zero mean value component of this solution is unique we get $\vfs u([\eta],\lb)=u(\vfs[\eta],\lb)$.
Thus
\begin{equation}\label{equivar_vfs-F}
\vfs{\cal F}_\lb[\eta]=\vfs([\eta]^{\ast H_u})=[\vfs\eta]^{\ast H_{\vfs u}}=
[\vfs\eta]^{\ast H_{u(\vfs[\eta],\lb)}}={\cal F}_\lb(\vfs[\eta]),
\end{equation}
and we are done.
\end{proof}

\subsection{Applications}

For the curvature equation \eqref{curveq}, the normalization
we adopted in the metric implies $K_0\equiv 4\pi$. The class of curvatures $K$ is restricted 
to $|\phi|_{H_0}^2$ for any $[\phi]\in\Phh$. Thus we are interested in studying problem 
\begin{equation}\label{curveq-K=phi}
\Delta u+|\phi|_0^2e^{2u}-4\pi=0.
\end{equation}
We start with a non-existence lemma which recovers the result in \cite{KW3}.
\begin{lem}\label{non-exist-lem}
Let $[\phi]$ be a class whose divisor set is $(k-2)x_0$, $k>2$, for some $x_0\in S^2$. 
Then there is no radially symmetric solution $u$ for equation \eqref{curveq-K=phi}, respect to the axis
of $S^2$ passing through $x_0$. In particular, $|\phi|_0^2$ is not the curvature of a rotationally symmetric
metric on $S^2$ pointwise conformal to the standard metric.
\end{lem}
\begin{proof}
If a radially symmetric solution $u$ for \eqref{curveq-K=phi} existed we could set 
$\eta=\phi^{\ast H_u}$ and have a solution for system \eqref{system-eta} with $\lb=4\pi$. 
On the other hand, computing the coefficients $\{b_j\}$ from \eqref{bj-coeff} in the stereographic coordinate
chart $w$ with south pole at $x_0$ would provide us with $b_1\neq 0$ and $b_j=0$, for $2\leq j\leq k-1$, thanks
to the symmetry of both $u$ and $|\phi|_0^2$. Following the same argument as in the proof of 
Lemma \ref{P_1-low-range} we conclude that $[\eta]\in\Pha_1$. This is an absurd due to Corollary \ref{eta-range-exist},
and the assumed solution $u$ does not exist.
\end{proof}

Before we go to the existence results on curvatures we first state a nice consequence of a standard differential
topology fact. The first part of this result was already known for more general functions \cite{KW1,Ms}.
The second part, though, seems to be new before \cite{Go2}.
\begin{thm}\label{exist-top-arg}
Let $[\phi]\in\Phh$ and $0<\lb<4\pi$. Then there exists at least one solution $u$ for \eqref{system-phi}. 
The cardinality of the set os solutions for $[\phi]$ equals the cardinality of the preimage 
${\cal F}_\lb^{-1}([\phi])$.
\end{thm}
\begin{proof}
Notice that in this range for $\lb$ the map ${\cal F}_\lb$ is defined for all $[\eta]\in\Pha$. Because of
Theorem 4.3 part (3) of \cite{Go2} we have that 
\begin{equation}
\lim_{\lb\to0^+}{\cal F}_\lb=\ast H_0.
\end{equation}
Set ${\cal F}_0=\ast H_0$. Since the family $\lb\to{\cal F}_\lb$ is continuous and ${\cal F}_0:\Pha\to\Phh$ is
a diffeomorphism, the maps ${\cal F}_\lb$ all have topological degree 1. In particular they are surjective.
Hence, any $[\phi]\in\Phh$ is of the form $[\eta]^{\ast H_{u([\eta],\lb)}}$, and so has a solution.
Because uniqueness of solutions holds for $[\eta]$ each solution for $[\phi]$ corresponds to exactly
one element of ${\cal F}_\lb^{-1}([\phi])$.
\end{proof}

Let $S\subset Iso(S^2)$ be a subgroup of the group of isometries of the euclidean sphere. Because of
Lemma \ref{F-lb-commutes-iso} any $S$-orbit of $\Pha$ is taken by ${\cal F}_\lb$ onto some $S$-orbit of
$\Phh$.

We first look at orbits which are unitary and isolated. If $\{[\eta]\}$ is such an orbit, meaning
that for any $[\tilde\eta]$ sufficiently close to $[\eta]$, $\{[\tilde\eta]\}$ is not an $S$-orbit, then
making $[\phi]={\cal F}_0[\eta]$ we get that $\{[\phi]\}$ is also a unitary and isolated $S$-orbit.
The continuity of the family ${\cal F}_\lb$ forces that ${\cal F}_\lb[\eta]=[\phi]$ for all $\lb$ in the
range of solutions for $[\eta]$. If $[\eta]$ is not in $\Pha_1$ we get that $|\phi|_0^2$ is the curvature
of a conformal metric. The next three theorems explore this idea for some symmetric classes $[\phi]$.
In the following we fix some arbitrary point $x_0\in S^2$ as reference, and denote by $l$ the axis 
passing through $\{x_0,-x_0\}$.
\begin{thm}\label{exist-a,b}
For positive integers $a,b$ with $a+b=k-2$ let ${\cal D}=a\,x_0+b\,(-x_0)$. Then the class $[\phi]\in\Phh$
with divisor ${\cal D}$ admits solution for $\lb=4\pi$. $|\phi|_0^2$ is the curvature of a metric conformal
to $g_0$.
\end{thm}
\begin{proof}
Let $S$ be the group of rotations about $l$. The classes in $\Phh$ with divisors given by $a\,x_0+b\,(-x_0)$
for $a,b\geq 0$ are the only classes fixed by the $S$-action. There are finite many of those, hence each one of them
is isolated. Finally, in case $a$ and $b$ are strictly positive the corresponding $[\eta]=[\phi]^{\ast H_0}$ 
is not in $\Pha_1$, hence $[\phi]$ admits solution in \eqref{system-phi} for $\lb=4\pi$. 
\end{proof}

\begin{thm}\label{exist-2a,n}
Let ${\cal D}=a\,x_0+a\,(-x_0)+{\cal E}$, $a>0$, where ${\cal E}$ is a divisor constructed as follows: 
choose an integer $n>2a$ and let ${\cal E}=\sum_{j=1}^n x_j$. The points $\{x_j\}_{1\leq j\leq n}$ lie in the
equator respect to $l$ and are evenly separated. The class $[\phi]$ with divisor ${\cal D}$ admits solution
at $\lb=4\pi$, and its $H_0$-norm squared is the curvature of a conformal metric.
\end{thm}
\begin{proof}
Let $\vf_1$ be a rotation of $2\pi/n$ about $l$. Clearly $[\phi]$ is fixed by the $\vfs_1$ action, but is not 
isolated. To accomplish that feature we consider $\vf_2$ the reflection respect to a plane $\beta$ containing 
$l$ and $x_1$, and $\vf_3$ the reflection respect to the equatorial plane. 
Let $S$ be the isometry subgroup generated by $\{\vf_1,\vf_2,\vf_3\}$. If $[\tilde\phi]$ is a class
fixed by $S$ then any zero of $[\tilde\phi]$ that is not $x_0$ or $-x_0$ repeats itself $n$ times along
a parallel, hence the number os those zeros is a multiple of $n$. Since there are $k-2=n+2a<2n$ zeros 
we conclude that either all zeros are in $\{x_0,-x_0\}$ or else there are $n$ zeros in a parallel and 
$2a$ zeros in $\{x_0,-x_0\}$. In the second case the $\vf_3$ invariance forces exactly $a$ zeros in each
of $x_0,-x_0$ and the parallel to be the equator. Still, there are two ways to inscribe the regular 
n-edge polygon inside the equatorial circle symmetrically respect to $\vf_2$. In either case we conclude that
the set of fixed points of $S$ is discrete and its ${\cal F}_0^{-1}$ image is disjoint of $\Pha_1$.
From this the assertion of the lemma follows.
\end{proof}

We now show a less trivial example of existence where the fixed points of the symmetry may not be fixed 
in the dynamics $\lb\mapsto{\cal F}_\lb$.
Let $S\subset Iso(S^2)$ be a subgroup. Let $Y^{1,0}\subset\Phh$ be a closed (compact, without boundary)
differentiable submanifold invariant for the $S$-action. Set $Y^{0,1}=(Y^{1,0})^{\ast H_0}$ its dual.
If ${\cal F}_\lb(Y^{0,1})\subset Y^{1,0}$ for all $\lb$ that makes sense and $Y^{0,1}\cap\Pha_1=\emptyset$ then
the same argument of Theorem \ref{exist-top-arg} applies since ${\cal F}_\lb:Y^{0,1}\to Y^{1,0}$ has degree
1, and for all $[\phi]\in Y^{1,0}$ there is a solution when $\lb=4\pi$.

\begin{thm}\label{existence_lem-ger}
Let ${\cal D}=a\,x_0+b\,(-x_0)+{\cal E}$ be a divisor with: $n>a>0$, $n>b>0$ and ${\cal E}$ is a divisor composed
by $mn$ zeros (counting multiplicity) evenly distributed in $m$ parallels (the parallels may not be pairwise 
distinct). Multiple zeros are allowed and the points $x_0$ and $-x_0$ might contain degenerated parallels. 
Then the class $[\phi]$ with this divisor has a solution in $\lb=4\pi$.
\end{thm}
\begin{proof}
Let $Y^{1,0}\subset\Phh$ be the set of all classes whose divisors are described by the lemma. We first show
it is a differentiable submanifold of $\Phh$, by exhibiting an embedding $i:\mathbb{CP}^m\to\Phh$ 
with image $Y^{1,0}$. Let $\zeta$ be the cannonical section of $L$ with $k$ zeros in $-x_0$, and let $z$ be 
a stereographic coordinate chart with $z(x_0)=0$. A representative of $[\phi]\in Y^{1,0}$ is of the 
form $\phi=z^ag(z)\zeta dz$ for some polynomial $g(z)$ of degree smaller than or equal to $k-2-a-b=mn$. 
Since the divisor of $g(z)\zeta\,dz$ is ${\cal E}+(a+b)(-x_0)$ a closer look at this structure reveals that 
$g(z)=h(z^n)$ where $h(v)$ is a polynomial in $v$ of degree no greater than $m$. The vector space of such 
polynomials is identified with $\C^{m+1}$, and there is an injective homomorphism $h\mapsto z^ah(z^n)\zeta\,dz$,
which passes to the projectivizations $\mathbb{CP}^m\to\Phh$, giving us the above mentioned embedding, and
$Y^{1,0}\simeq\mathbb{CP}^m$.

Let $Y^{0,1}=(Y^{1,0})^{\ast H_0}$. Clearly $Y^{0,1}\cap\Pha_1=\emptyset$ because $a,b\neq0$. 
We claim that ${\cal F}_\lb(Y^{0,1})\subset Y^{1,0}$ for all $\lb$ in an open range containing $(0,4\pi]$.
To see that we'd rather index the space $Y^{1,0}\equiv Y^{1,0}_{a,b}$ and consider all such 
submanifolds $Y^{1,0}_{a',b'}$ for $a'$ an integer, $0\leq a'<n$ and $b'=(k-2-a')\text{mod}\,n$. 
Notice that in the space $Y^{1,0}_{a',b'}$ the number $m'$ of ``parallels'' may differ from $m$, and
the cases $a'=0$ or  $b'=0$ occur, but that does not matter for our argument. 

Setting $S$ as the subgroup generated by the rotation of $2\pi/n$ about $l$ it becomes clear that
$[\phi]\in\Phh$ is $S$-invariant if and only if $[\phi]\in\cup_{0\leq a'<n}Y^{1,0}_{a',b'}$. 
Hence the image ${\cal F}_\lb(Y^{0,1})$ is contained in the union of the components $Y^{1,0}_{a',b'}$, 
each of them being a copy of some complex projective space and pairwise disjoint. The continuity of 
$\lb\mapsto{\cal F}_\lb$ then precludes the image ${\cal F}_\lb(Y^{0,1})$ from leaving the original copy $Y^{1,0}_{a,b}$. 
This concludes the lemma. 
\end{proof}

\begin{rem}
It is interesting to look at the actual functions $K=2|\phi|_0^2$ in the previous lemmas. Let
 $z=z_{-x_0}$ a stereographic coordinate that vanishes at $x_0$, and $\zeta$ a holomorphic section 
with divisor $k(-x_0)$. A general holomorphic $\phi=g\,\zeta_L dz$ has norm
\begin{equation}\label{phi-squared}
   |\phi|_{H_0}^2=2\pi\frac{|a_0+a_1z+\dots+a_{k-2}z^{k-2}|^2}{(1+|z|^2)^{k-2}}
\end{equation}
for complex constants $a_j$, $0\leq j\leq k-2$ (the factor $2\pi$ is due to the normalization $|S^2|=1$
on the tangent bundle). Hence, the following functions are curvatures in the conformal structure of $g_0$:
\begin{enumerate}
\item $K(z)=|z|^{2a}(1+|z|^2)^{2-k}$ if $0<a<k-2$ (Theorem \ref{exist-a,b});
\item $K(z)=|z|^{2a}|z^n-1|^2(1+|z|^2)^{2-k}$ if $n>2a>0$ and $2a+n=k-2$ (Theorem \ref{exist-2a,n});
\item $K(z)=|z|^{2a}|z^n-q_1|^2|z^n-q_2|^2\ldots|z^n-q_m|^2(1+|z|^2)^{2-k}$ for arbitrary complex
numbers $q_1,\ldots,q_m$, if $n>a>0$, $a+mn<k-2$ and $(k-2-a)\,\text{mod}\,n>0$ (Theorem \ref{existence_lem-ger});
\end{enumerate}
Notice in the above that, except in very few cases, the functions $K$ are not symmetric about the origin,
and the existence results of \cite{Ms} do not apply directly. Since all such functions have zeros one
cannot use the results on \cite{CY}.
\end{rem}
\begin{rem}
At this point of the research it seems to us that the holomorphicity of $\phi$ is not the key to obtaining
the existence results on \eqref{curveq}, but only the behaviour of $|\phi|_0^2$ around
its zeros. In a future work we intend to show existence results for \eqref{curveq} for a larger
class of \emph{smooth} functions $K\geq0$ with finite many zeros with even degrees, and spread over $S^2$
in a suitable way.
\end{rem}

\end{document}